\documentclass[11pt]{article} 
\usepackage{tikz}
\usepackage{amsmath,amsthm,amssymb}
\usepackage{mathrsfs}
\usepackage{graphicx}
\usepackage{subfig}
\usepackage{color,xcolor}
\usepackage{enumitem}
\usepackage{float}
\usepackage{pifont}

\usepackage{tabularx}
\usepackage{booktabs}
\usepackage{array}
\usepackage{multirow,multicol}
\usepackage{longtable}
\usepackage{makecell}
\usepackage{anyfontsize}
\usepackage{microtype}
\usepackage{geometry}
\usepackage[no-math]{fontspec}

\usepackage{hyperref}
\hypersetup{
  colorlinks=true,
  linkcolor=black,
  citecolor=blue,
  filecolor=blue,
  urlcolor=blue
}

\allowdisplaybreaks

\numberwithin{equation}{section}
\numberwithin{figure}{section}
\numberwithin{table}{section}

\usepackage{algorithm}
\usepackage{algpseudocode}
\floatname{algorithm}{Algorithm}
\algrenewcommand\algorithmicrequire{\textbf{Input:}}
\algrenewcommand\algorithmicensure{\textbf{Output:}}

\setlist{nolistsep}

\graphicspath{{./figure/}{./figures/}}

\newcolumntype{L}{X}
\newcolumntype{C}{>{\centering \arraybackslash}X}
\newcolumntype{R}{>{\raggedleft \arraybackslash}X}
\newcolumntype{P}[1]{>{\centering \arraybackslash}p{#1}}

\theoremstyle{plain}
\newtheorem{definition}{Definition}[section]

\newtheorem{lemma}{Lemma}[section]
\newtheorem{theorem}{Theorem}[section]

\renewenvironment{abstract}
{\par\noindent\ignorespaces}
{\par\noindent\ignorespacesafterend}
\usepackage{indentfirst}

\title{The Symmetry and Linear Stability of Convex 1+5 Coorbital Central Configurations with Homogeneous Potential}
\author {Yiyang Deng and Jiangtao Xu}
\date{}

\begin{document}

\maketitle

\begin{abstract}
\noindent\textbf{Abstract:} For the planar Newtonian 1+N-body problem when the N masses tend to zero, the corresponding relative equilibria become coorbital around the dominant mass. In this work, we focus on convex central configurations in the planar 1+N coorbital problem. For the 1+5 coorbital problem with the homogeneous potential, we prove that any convex coorbital central configuration with symmetric masses must have an axis of symmetry. Furthermore, under explicit restrictions on the angular variables in a homogeneous potential, we prove the linear stability of both convex symmetric 1+5 and convex 1+N coorbital central configurations.
\medskip

\end{abstract}


\section{Introduction}

Central configurations plays a very important role in N-body problems. Central configurations are those for which the total Newtonian acceleration of every body equals a constant multiplied by its position vector relative to the configuration's center of mass. The study of central configurations and relative equilibria provides an avenue for progress into the N-body problem. Euler \cite{Euler1767} and Lagrange \cite{Lagrange1772} characterized the relative equilibria for the Newtonian three-body problem. The collinear N-body configurations studied by Moulton \cite{Moulton1910}, who showed there is a unique (up to scaling) central configuration for any ordering of N positive masses on a line.

One of the most basic questions is whether there are finitely many equivalence classes of central configurations for each choice of positive masses.
It has been resolved for the Newtonian four-body problem \cite{Hampton2005}, the four-vortex problem \cite{Hampton2009, Yu2023}, partially for the Newtonian five-body problem \cite{Albouy2012} and the five-vortex problem \cite{Yu2025}.

In celestial mechanics, the 1+N-body problem describes the gravitational interaction between a dominant primary mass and N infinitesimal masses.
It serves as a fundamental model for investigating planetary ring systems and coorbital motion.
As early as the work of Maxwell \cite{Maxwell1859}, this framework was employed to characterize structures resembling Saturn’s rings.
Subsequent studies have further highlighted its physical relevance in coorbital satellite systems, such as the well-known Janus–Epimetheus coorbital moons of Saturn \cite{Yoder1983}, which travel on an interesting horseshoe orbit.
The preprint by Hall \cite{Hall1988} framed the problem more generally within the context of the study of central configurations.
These configurations, characterized by intricate geometric properties and a close relationship with the classical theory of central configurations, continue to attract considerable research attention.
In 1988, Salo and Yoder \cite{Salo1988} numerically extensively investigated the configurations and dynamics of the $1+N$-body problem for $N<10$ with identical infinitesimal masses.
Verrier and McInnes \cite{Verrier2014} analyzed horseshoe orbits by numerical continuations from relative equilibria in the $1+2$, $1+3$, and $1+4$ coorbital problems.

Deng, Hampton and Wang \cite{Deng2022} extended the approach of Renner and Sicardy \cite{Renner2004} in considering the antisymmetry of the mass coefficient matrix of the defining equations for these central configurations.
This manuscript focuses on the $1+N$-body problem for $N = 5$. Numerically it \cite{Salo1988} appears that there is a single $1+N$-body configuration with $N$ equal masses (the regular $N$-gon around the central mass) for $N \ge 9$, but lower values of $N$ exhibit more complexity.
Hall \cite{Hall1988} proved the uniqueness of the $1+N$ equal small masses central configuration for $N \ge e^{27000}$, which was improved to $N \ge e^{73}$ by Casasayas, Llibra and Nunes \cite{Casasayas1994}.


For a central mass that is large but finite relative to the coorbital infinitesimal masses, some additional results have been derived, all of which assume equal infinitesimal masses arranged on a regular polygonal ring.
Maxwell \cite{Maxwell1859} showed that for a sufficiently large central mass this configuration would be stable, although his analysis was incorrect for $N < 7$.
For $N \ge 7$, Moeckel \cite{Moeckel1994} found a necessary and sufficient criterion for the linear stability of relative equilibria in the 1+N body problem with small but not necessarily equal masses, and proved that Maxwell's ring is linearly stable if and only if $N \ge 7$. Roberts subsequently found a bifurcation value $M_{bif}$ and showed that this configuration is linearly stable precisely when the central mass satisfies $M > M_{bif}$.
Some related results on stability for regular polygonal $1+N$ are proved by Xu \cite{Xu2013}.

Very little is known about $1+N$-body central configurations with unequal masses for $N \ge 4$.
The $1+N$-body problem can be extended to other potentials, such as the point vortex model \cite{Barry2012}.
There are some results on $1+3$  \cite{Barry2016} and  $1+4$ \cite{Hoyer2026,Oliveira2013,Deng2019,Oliveira2020} planar vortex central configurations.
The $1+3$ problem already possesses a remarkably rich structure, which has been well-studied in the Newtonian case \cite{Corbera2011, Corbera2015}.
The central configurations for the Newtonian $1+4$ problem were rigorously determined by Albouy and Fu \cite{Albouy2009}. For $1+5$ problem, Su and Deng \cite{Su2022} studied the relationship between the masses of 5 satellites and given symmetric configurations when one satellite locate at the symmetry axis.

Renner and Sicardy studied the stationary configurations of coorbital satellites with arbitrary small masses. They proved that the linear stability of a coorbital relative equilibrium is determined by the eigenvalues of the matrix \cite{Renner2004} $A=M^{-1}H$, where H is the Hessian matrix of the effective potential.

In this paper we are interested in the central configurations of the planar $1 + 5$ body coorbital problem with symmetric masses.
In section 2, we recall the $1+N$ coorbital problem central configuration equations.
In section 3, we discuss the symmetry of convex $1+5$ coorbital central configurations in homogeneous potential.
In the last section 4, we consider the linear stability of the convex symmetric 1 + 5 coorbital central configurations with homogeneous potential and extend the similar result to convex 1 + N coorbital central configurations.

\section{1+N Coorbital Problem Central Configuration Equations}

In this paper, we consider the 1+N coorbital problem under the homogeneous potential. The derivation of the central configuration equations in \cite{Corbera2015} extends easily to this setting. Namely, for a configuration of $N$ infinitesimal masses to form a $1 + N$ planar central configuration they must lie on a common circle, which can be scaled to be the unit circle. For notational convenience, we introduce polar coordinates, denoting by $\theta_i$ the angular position of the $i$-th body, and adopt the shorthand $\theta_{ij} = \theta_i - \theta_j$.

With this notation, the governing equations can be expressed in terms of the following function:
\begin{equation*}
f_{ij}=\sin \left(\theta_{ij}\right)\left(\frac{1}{r_{ij}^{s}}-1\right),
\end{equation*}
where $s$ denotes the exponent of the interaction potential $s=3$ corresponds to the Newtonian gravitational case, while $s=2$ represents the point-vortex setting, and $r_{ij}$ is the distance between point $i$ and point $j$. We allow a slight abuse of notation by also treating $f$ as a single-variable function, written in the form:
\begin{equation*}
f(\theta)=\sin (\theta)\left(2^{-s}|\sin \theta / 2|^{-s}-1\right).
\end{equation*}

Note that, for points located on the unit circle, the mutual distances admit several equivalent representations:
\begin{equation*}
r_{ij}=\sqrt{2-2 \cos \left(\theta_{ij}\right)}=2\left|\sin \left(\frac{\theta_{ij}}{2}\right)\right|.
\end{equation*}

Let $f_{ij}$ denote the mass coefficient matrix of the coorbital system, where the entries $f_{ij}$ are defined accordingly with $f_{ii} = 0$. Then the 1+N coorbital central configuration equations can be written as:
\begin{equation}\label{MainEq}
Fm=0,
\end{equation}
where $m = (m_1, m_2, \dots, m_N)^T.$

Since $F$ is real and antisymmetric, it has purely imaginary eigenvalues, and the dimension of the kernel of $F$ has the same parity as $N$.

According to the paper \cite{Hall1988,Deng2022}, when $s>1$ and $s \neq 2$, we can view the coorbital relative equilibrium equations as conditions to have a critical point of the effective potential:
$$V = \sum_{1\leq i < j \leq N} m_i m_j \left( \frac{1}{(s-2) r_{ij}^{s-2}} + \frac{r_{ij}^2}{2} \right).$$
It is easy to check that
$$\frac{\partial V}{\partial \theta_i} = - m_i \sum_{j \neq i} m_j f_{ij}.$$

This can be extended to the point vortex case \cite{Barry2012, Barry2016} by using the potential
$$V = \sum_{1\leq i < j \leq N} m_i m_j \left( \log(r_{ij}) + \frac{r_{ij}^2}{2} \right).$$
In this paper, we refer to this vortex potential as the $s = 2$ case.

Let $M = diag (m_1, m_2, \dots, m_N)$,  the coorbital central configuration equations can be written as
$$F m = - M^{-1} \nabla V = 0.$$

\section{Convex 1+5 Coorbital Central Configurations with Homogeneous Potential}

\begin{figure}[htbp]
\centering
\begin{tikzpicture}[scale=3]

\draw[thick] (0,0) circle (1);

\draw[thick] (-1,0) -- (1,0);
\draw[thick] (0,-1) -- (0,1);

\def\thetaone{75}
\def\thetatwo{35}
\def\thetathree{350}
\def\thetafour{321}
\def\thetafive{300}

\draw[thick] (0,0) -- (\thetaone:1);
\draw[thick] (0,0) -- (\thetatwo:1);
\draw[thick] (0,0) -- (\thetathree:1);
\draw[thick] (0,0) -- (\thetafour:1);
\draw[thick] (0,0) -- (\thetafive:1);

\fill (0,0) circle (2pt);
\fill (\thetaone:1) circle (1pt) node at (\thetaone:1.2) {$\theta_1$};
\fill (\thetatwo:1) circle (1pt) node at (\thetatwo:1.2) {$\theta_2$};
\fill (\thetathree:1) circle (1pt) node at (\thetathree:1.2) {$\theta_3$};
\fill (\thetafour:1) circle (1pt) node at (\thetafour:1.2) {$\theta_4$};
\fill (\thetafive:1) circle (1pt) node at (\thetafive:1.2) {$\theta_5$};

\end{tikzpicture}

\caption{1+5 Convex Coorbital Configuration}
\label{convex}

\end{figure}
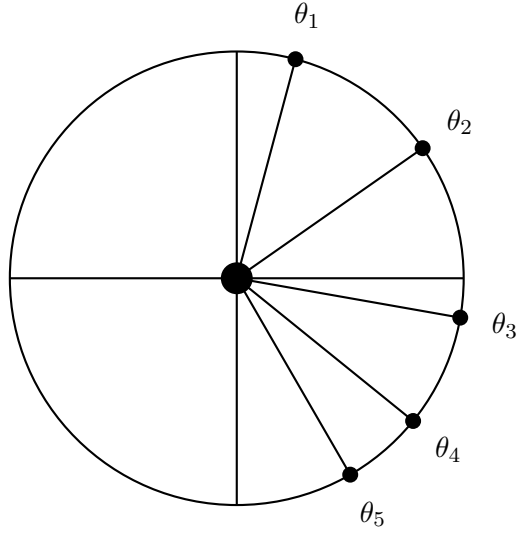

A convex coorbital configuration means that the convex hull of the coorbital points does not contain the center of the circle, as shown in the Figure \ref{convex}.
Deng, Hampton and Wang \cite{Deng2022} proved that there exist symmetric central configurations with asymmetric positive masses when $N = 4, 6, 8$ for $1+N$ coorbital problem.
In contrast to the result, they proved that a $1 + 5$ convex coorbital central configuration with $m_1 = m_5 > 0$ and $m_2 = m_4 > 0$ must be symmetric when $s=3$ (Newtonian potential).
But their method is not available in homogeneous potentials.

In this section, we conduct a systematic analysis of the symmetry conditions for $1+5$ convex coorbital configurations in homogeneous potential ($s>1$). 
We get the following result:

\begin{lemma}\label{mf}
For 1+5 coorbital configurations with homogeneous potentials, if $\theta_{ij}, \theta_{kl}$ are contained in $[0, 2\pi],$ then the quantity
$$\frac{\sin (\theta_{ij} ) r_{kl}^{s} (1-r_{ij}^{s})-\sin (\theta_{kl}) r_{ij}^{s} (1-r_{kl}^{s})}{r_{ij}^{s} r_{kl}^{s}}$$
can be written as
$$2\cos\frac{\theta_{ij}+\theta_{kl}}{2}\sin\frac{\theta_{kl}-\theta_{ij}}{2} + \frac{1}{2^{s-1}} \left( \frac{\cos\frac{\theta_{ij}}{2}}{\sin^{s-1}\frac{\theta_{ij}}{2}} - \frac{\cos\frac{\theta_{kl}}{2}}{\sin^{s-1}\frac{\theta_{kl}}{2}} \right).$$
\end{lemma}
\begin{proof}
Under the assumptions, we can write the distances in terms of trigonometric functions and factor out some common terms:
\begin{equation*}
\begin{aligned}
& \sin (\theta_{ij} ) r_{kl}^{s} (1-r_{ij}^{s})-\sin (\theta_{kl}) r_{ij}^{s} (1-r_{kl}^{s})\\
= & 2^{s}\left(1-2^{s}\sin ^{s} \frac{\theta_{ij}}{2}\right) \sin ^{s} \frac{\theta_{kl}}{2} \sin \frac{\theta_{ij}}{2} \cos \frac{\theta_{ij}}{2}-2^{s}\left(1-2^{s}\sin ^{s} \frac{\theta_{kl}}{2}\right) \sin ^{s} \frac{\theta_{ij}}{2} \sin \frac{\theta_{kl}}{2} \cos \frac{\theta_{kl}}{2} \\
= & 2^{s+1}\sin \frac{\theta_{ij}}{2} \sin \frac{\theta_{kl}}{2}\left[\left(1-2^{s}\sin ^{s} \frac{\theta_{ij}}{2}\right) \sin^{s-1}\frac{\theta_{kl}}{2} \cos \frac{\theta_{ij}}{2}-\left(1-2^{s}\sin ^{s} \frac{\theta_{kl}}{2}\right) \sin^{s-1}\frac{\theta_{ij}}{2} \cos \frac{\theta_{kl}}{2}\right] \\
& \text{and now we continue with some elementary trigonometric identities (sum-to-product):} \\
 =  & 2^{s+1}\sin \frac{\theta_{ij}}{2} \sin \frac{\theta_{kl}}{2}\bigg(2^{s}\sin^{s-1}\frac{\theta_{ij}}{2}\sin^{s-1}\frac{\theta_{kl}}{2}\cos\frac{\theta_{ij}+\theta_{kl}}{2}\sin\frac{\theta_{kl}-\theta_{ij}}{2} + \\
 &\sin^{s-1}\frac{\theta_{kl}}{2}\cos\frac{\theta_{ij}}{2}-\sin^{s-1}\frac{\theta_{ij}}{2}\cos\frac{\theta_{kl}}{2} \bigg)
\end{aligned}
\end{equation*}

Furthermore, we have
\begin{equation*}
\begin{aligned}
&\frac{\sin (\theta_{ij} ) r_{kl}^{s} (1-r_{ij}^{s})-\sin (\theta_{kl}) r_{ij}^{s} (1-r_{kl}^{s})}{r_{ij}^{s} r_{kl}^{s}} \\
= & \frac{2^{s+1}\sin \frac{\theta_{ij}}{2} \sin \frac{\theta_{kl}}{2} (2^{s}\sin^{s-1}\frac{\theta_{ij}}{2}\sin^{s-1}\frac{\theta_{kl}}{2}\cos\frac{\theta_{ij}+\theta_{kl}}{2}\sin\frac{\theta_{kl}-\theta_{ij}}{2}
+\sin^{s-1}\frac{\theta_{kl}}{2}\cos\frac{\theta_{ij}}{2}-\sin^{s-1}\frac{\theta_{ij}}{2}\cos\frac{\theta_{kl}}{2} )}{2^{2s}\sin ^{s} \frac{\theta_{ij}}{2} \sin ^{s} \frac{\theta_{kl}}{2} }\\
= & 2\cos\frac{\theta_{ij}+\theta_{kl}}{2}\sin\frac{\theta_{kl}-\theta_{ij}}{2} + \frac{\sin^{s-1}\frac{\theta_{kl}}{2}\cos\frac{\theta_{ij}}{2}-\sin^{s-1}\frac{\theta_{ij}}{2}\cos\frac{\theta_{kl}}{2}}{2^{s-1}\sin ^{s-1} \frac{\theta_{ij}}{2} \sin ^{s-1} \frac{\theta_{kl}}{2} } \\
= & 2\cos\frac{\theta_{ij}+\theta_{kl}}{2}\sin\frac{\theta_{kl}-\theta_{ij}}{2} + \frac{1}{2^{s-1}} \left( \frac{\cos\frac{\theta_{ij}}{2}}{\sin^{s-1}\frac{\theta_{ij}}{2}} - \frac{\cos\frac{\theta_{kl}}{2}}{\sin^{s-1}\frac{\theta_{kl}}{2}} \right).
\end{aligned}
\end{equation*}
\end{proof}

\begin{theorem}
Let $s$ denote the exponent of the interaction potential. When $s>1$, any $1+5$ coorbital central configuration satisfying the ordering
$-\frac{\pi}{2} < \theta_5 < \theta_4 < \theta_3 < \theta_2 < \theta_1 < \frac{\pi}{2}$ together with the mass symmetry conditions $m_1 = m_5$ and $m_2 = m_4$, necessarily admits a symmetry axis.
\end{theorem}

\begin{proof}

Similarly, in the mass-coefficient matrix, the $1+5$ configuration yields two independent effective equations, which can be extracted as follows:
\begin{equation}\label{keyeq}
\begin{aligned}
&m_{1}(f_{13}-f_{35})+m_{2}(f_{23}-f_{34})=0 \\
&m_{1}(f_{12} + f_{14} - f_{25} -f_{45} ) - m_{3}(f_{23}-f_{34})=0
\end{aligned}
\end{equation}

In the case of homogeneous potentials, the transformation of the effective equations undergoes certain modifications. In particular, by using lemma \ref{mf}, we obtain:
\begin{equation*}
\begin{aligned}
  & m_{1}(f_{13}-f_{35})+m_{2}(f_{23}-f_{34}) \\
 =& m_{1}\left [\frac{\sin\theta_{13}r_{35}^{s}(1-r_{13}^{s})}{r_{13}^{s}r_{35}^{s}}-\frac{\sin\theta_{35}r_{13}^{s}(1-r_{35}^{s})}{r_{13}^{s}r_{35}^{s}}\right]
    +m_{2}\left[\frac{\sin\theta_{23}r_{34}^{s}(1-r_{23}^{s})}{r_{23}^{s}r_{34}^{s}}-\frac{\sin\theta_{34}r_{23}^{s}(1-r_{34}^{s})}{r_{23}^{s}r_{34}^{s}}\right]\\
 = & m_1 \left[2\sin\frac{\theta_{35}-\theta_{13}}{2}\cos\frac{\theta_{13}+\theta_{35}}{2}  +  \frac{1}{2^{s-1}}\left(\frac{\cos\frac{\theta_{13}}{2}}{\sin ^{s-1} \frac{\theta_{13}}{2}}  - \frac{\cos\frac{\theta_{35}}{2}}{\sin ^{s-1} \frac{\theta_{35}}{2}} \right)\right] \\
  +& m_2 \left[2\sin\frac{\theta_{34}-\theta_{23}}{2}\cos\frac{\theta_{23}+\theta_{34}}{2}+ \frac{1}{2^{s-1}}\left(\frac{\cos\frac{\theta_{23}}{2}}{\sin ^{s-1} \frac{\theta_{23}}{2}}  - \frac{\cos\frac{\theta_{34}}{2}}{\sin ^{s-1} \frac{\theta_{34}}{2}} \right)\right].
\end{aligned}
\end{equation*}
Next, we claim that $\theta_{13} = \theta_{35}$ and $\theta_{23} = \theta_{34}$ is the only solution of the equations \ref{keyeq}.

If $\theta_{13} \neq \theta_{35}$ and $\theta_{23} \neq \theta_{34}$, we have the following cases:
\begin{equation*}
 \begin{aligned}
\textbf{(1)} \left\{\begin{matrix}
\theta_{13}>\theta_{35} \\
\theta_{23}>\theta_{34}
\end{matrix}\right.
\textbf{(2)} \left\{\begin{matrix}
 \theta_{13}<\theta_{35} \\
 \theta_{23}<\theta_{34}
\end{matrix}\right.
\textbf{(3)} \left\{\begin{matrix}
 \theta_{13}=\theta_{35} \\
 \theta_{23}\neq\theta_{34}
\end{matrix}\right.
\textbf{(4)} \left\{\begin{matrix}
  \theta_{13}\neq\theta_{35} \\
  \theta_{23}=\theta_{34}
\end{matrix}\right.
\textbf{(5)} \left\{\begin{matrix}
 \theta_{13}>\theta_{35} \\
 \theta_{23}<\theta_{34}
\end{matrix}\right.
\textbf{(6)} \left\{\begin{matrix}
  \theta_{13}<\theta_{35}\\
  \theta_{23}>\theta_{34}
\end{matrix}\right..
\end{aligned}
\end{equation*}

For easy to discuss above cases, we set
$$f(x) = \frac{\cos x}{\sin^{s-1} x}.$$
The derivative of the function $f(x)$ with respect to the variable $x$ is obtained as
$$f'(x) = - \frac{ \sin^2  x + (s-1) \cos^2 x}{\sin^s x}.$$
When $s > 1$, we have $f'(x) < 0 $ for $0 < x < \pi$.

\textbf{Case (1).} $\theta_{13} > \theta_{35}$ and $\theta_{23} > \theta_{34}$

When $\theta_{13} > \theta_{35}$ and $\theta_{23} > \theta_{34}$, it is obvious that $\sin\frac{\theta_{35}-\theta_{13}}{2} < 0$ and $\sin\frac{\theta_{34} - \theta_{23}}{2}  <0,$ which implies the sign of the terms
\begin{equation*}
\begin{aligned}
\sin\frac{\theta_{35}-\theta_{13}}{2}\cos\frac{\theta_{13}+\theta_{35}}{2}, ~~~ \sin\frac{\theta_{34}-\theta_{23}}{2}\cos\frac{\theta_{23}+\theta_{34}}{2}.
\end{aligned}
\end{equation*}
are both negative for the convex configurations.

According to the monotonicity of the function $f(x)$, we have
$$\frac{\cos\frac{\theta_{13}}{2}}{\sin ^{s-1} \frac{\theta_{13}}{2}}  - \frac{\cos\frac{\theta_{35}}{2}}{\sin ^{s-1} \frac{\theta_{35}}{2}} <0,~~
 \frac{\cos\frac{\theta_{23}}{2}}{\sin ^{s-1} \frac{\theta_{23}}{2}}  - \frac{\cos\frac{\theta_{34}}{2}}{\sin ^{s-1} \frac{\theta_{34}}{2}} < 0.$$
Therefore, we get that
\begin{equation*}
\begin{aligned}
& m_{1}(f_{13}-f_{35})+m_{2}(f_{23}-f_{34}) \\
= & m_1 \left[2\sin\frac{\theta_{35}-\theta_{13}}{2}\cos\frac{\theta_{13}+\theta_{35}}{2}  +  \frac{1}{2^{s-1}}\left(\frac{\cos\frac{\theta_{13}}{2}}{\sin ^{s-1} \frac{\theta_{13}}{2}}  - \frac{\cos\frac{\theta_{35}}{2}}{\sin ^{s-1} \frac{\theta_{35}}{2}} \right) \right]\\
&+ m_2 \left[2\sin\frac{\theta_{34}-\theta_{23}}{2}\cos\frac{\theta_{23}+\theta_{34}}{2}+ \frac{1}{2^{s-1}}\left(\frac{\cos\frac{\theta_{23}}{2}}{\sin ^{s-1} \frac{\theta_{23}}{2}}  - \frac{\cos\frac{\theta_{34}}{2}}{\sin ^{s-1} \frac{\theta_{34}}{2}} \right)\right]<0.
\end{aligned}
\end{equation*}
It is impossible.

\textbf{Case (2).} $\theta_{13} < \theta_{35}$ and $\theta_{23} < \theta_{34}$

It is similar to case 1, we have
\begin{equation*}
\begin{aligned}
& m_{1}(f_{13}-f_{35})+m_{2}(f_{23}-f_{34}) \\
= & m_1 \left[2\sin\frac{\theta_{35}-\theta_{13}}{2}\cos\frac{\theta_{13}+\theta_{35}}{2}  +  \frac{1}{2^{s-1}}\left(\frac{\cos\frac{\theta_{13}}{2}}{\sin ^{s-1} \frac{\theta_{13}}{2}}  - \frac{\cos\frac{\theta_{35}}{2}}{\sin ^{s-1} \frac{\theta_{35}}{2}} \right)\right] \\
 & + m_2 \left[2\sin\frac{\theta_{34}-\theta_{23}}{2}\cos\frac{\theta_{23}+\theta_{34}}{2}+ \frac{1}{2^{s-1}}\left(\frac{\cos\frac{\theta_{23}}{2}}{\sin ^{s-1} \frac{\theta_{23}}{2}}  - \frac{\cos\frac{\theta_{34}}{2}}{\sin ^{s-1} \frac{\theta_{34}}{2}} \right)\right]>0.
\end{aligned}
\end{equation*}
It is still impossible.

\textbf{Case (3).} $\theta_{13} = \theta_{35}$ and $\theta_{23} \neq \theta_{34}$

Under the assumptions, we get
\begin{equation*}
\begin{aligned}
  & m_{1}(f_{13}-f_{35})+m_{2}(f_{23}-f_{34}) =  m_{2}(f_{23}-f_{34})\\
 =& m_2 \left[2\sin\frac{\theta_{34}-\theta_{23}}{2}\cos\frac{\theta_{23}+\theta_{34}}{2}+ \frac{1}{2^{s-1}}\left(\frac{\cos\frac{\theta_{23}}{2}}{\sin ^{s-1} \frac{\theta_{23}}{2}}  - \frac{\cos\frac{\theta_{34}}{2}}{\sin ^{s-1} \frac{\theta_{34}}{2}} \right)\right] \neq 0.
\end{aligned}
\end{equation*}
It is impossible.

\textbf{Case (4).} $\theta_{13} \neq \theta_{35}$ and $\theta_{23} = \theta_{34}$

By the above suppose, we have
\begin{equation*}
\begin{aligned}
  & m_{1}(f_{13}-f_{35})+m_{2}(f_{23}-f_{34}) = m_{1}(f_{13}-f_{35}) \\
= & m_1 \left[2\sin\frac{\theta_{35}-\theta_{13}}{2}\cos\frac{\theta_{13}+\theta_{35}}{2}  +  \frac{1}{2^{s-1}}\left(\frac{\cos\frac{\theta_{13}}{2}}{\sin ^{s-1} \frac{\theta_{13}}{2}}  - \frac{\cos\frac{\theta_{35}}{2}}{\sin ^{s-1} \frac{\theta_{35}}{2}} \right)\right] \neq 0.
\end{aligned}
\end{equation*}
It is impossible.

Then, we just need to consider the following two cases: (5) $\theta_{13}>\theta_{35}$ and $\theta_{23}<\theta_{34}$; (6) $\theta_{13}<\theta_{35}$ and $\theta_{23}>\theta_{34}.$

For the other effective equation, we proceed in a similar manner:
\begin{equation*}
\begin{aligned}
 & m_{1}(f_{12} + f_{14} -f_{25} - f_{45})+ m_{3}(f_{34} - f_{23}) \\
= & m_{1} \left(\frac{\sin(\theta_{12})r_{45}^{s}(1-r_{12}^{s})-\sin(\theta_{45})r_{12}^{s}(1-r_{45}^{s})}{r_{12}^{s}r_{45}^{s}} + \frac{\sin(\theta_{14})r_{25}^{s}(1-r_{14}^{s})-\sin(\theta_{25}) r_{14}^{s}(1-r_{25}^{s})}{r_{14}^{s} r_{25}^{s}} \right)\\
 &+m_{3}\frac{\sin(\theta_{34})r_{23}^{s}(1-r_{34}^{s})-\sin(\theta_{23})r_{34}^{s}(1-r_{23}^{s})}{r^{3}_{23}r^{3}_{34}}=0.
\end{aligned}
\end{equation*}

The first term of the above equation is
\begin{equation*}
\begin{aligned}
 &\frac{\sin \left(\theta_{12}\right) r_{45}^{s}\left(1-r_{12}^{s}\right)-\sin \left(\theta_{45}\right) r_{12}^{s}\left(1-r_{45}^{s}\right)}{r_{12}^{s} r_{45}^{s}}
 +\frac{\sin \left(\theta_{14}\right) r_{25}^{s}\left(1-r_{14}^{s}\right)-\sin \left(\theta_{25}\right) r_{14}^{s}\left(1-r_{25}^{s}\right)}{r_{14}^{s} r_{25}^{s}} \\
 = & 2\cos\frac{\theta_{12}+\theta_{45}}{2}\sin\frac{\theta_{45}-\theta_{12}}{2} + \frac{\sin^{s-1}\frac{\theta_{45}}{2}\cos\frac{\theta_{12}}{2}-\sin^{s-1}\frac{\theta_{12}}{2}\cos\frac{\theta_{45}}{2}}{2^{s-1}\sin ^{s-1} \frac{\theta_{12}}{2} \sin ^{s-1} \frac{\theta_{45}}{2} } + \\
  & 2\cos\frac{\theta_{14}+\theta_{25}}{2}\sin\frac{\theta_{25}-\theta_{14}}{2} + \frac{\sin^{s-1}\frac{\theta_{25}}{2}\cos\frac{\theta_{14}}{2}-\sin^{s-1}\frac{\theta_{14}}{2}\cos\frac{\theta_{25}}{2}}{2^{s-1}\sin ^{s-1} \frac{\theta_{14}}{2} \sin ^{s-1} \frac{\theta_{25}}{2} } \\
 = & 2\sin\frac{\theta_{45}-\theta_{12}}{2} \left( \cos\frac{\theta_{12}+\theta_{45}}{2} + \cos\frac{\theta_{14}+\theta_{25}}{2} \right)  + \\
  &\frac{1}{2^{s-1}}\left(\frac{\cos\frac{\theta_{12}}{2}}{\sin ^{s-1} \frac{\theta_{12}}{2}}  - \frac{\cos\frac{\theta_{45}}{2}}{\sin ^{s-1} \frac{\theta_{45}}{2}} \right) + \frac{1}{2^{s-1}}\left(\frac{\cos\frac{\theta_{14}}{2}}{\sin ^{s-1} \frac{\theta_{14}}{2}}  - \frac{\cos\frac{\theta_{25}}{2}}{\sin ^{s-1} \frac{\theta_{25}}{2}} \right).
\end{aligned}
\end{equation*}

It is easy to see that
\begin{equation*}
\begin{aligned}
\cos \left(\frac{\theta_{12}}{2}+\frac{\theta_{45}}{2}\right)+\cos \left(\frac{\theta_{14}}{2}+\frac{\theta_{25}}{2}\right) & =\cos \left(\frac{\theta_{15}}{2}-\frac{\theta_{24}}{2}\right)+\cos \left(\frac{\theta_{15}}{2}+\frac{\theta_{24}}{2}\right) \\
& =2 \cos \frac{\theta_{15}}{2} \cos \frac{\theta_{24}}{2}.
\end{aligned}
\end{equation*}

When the configuration is convex, the sign of $\cos \frac{\theta_{15}}{2} \cos \frac{\theta_{24}}{2}$ is positive.

From $\theta_{13}>\theta_{35}$ and $\theta_{23}<\theta_{34}$, it implies $\theta_{12} > \theta_{45}$. Because $\theta_{12} - \theta_{45} = \theta_{14} - \theta_{25}$, we also have $\theta_{14}>\theta_{25}.$

Similarly, by $\theta_{13} < \theta_{35}$ and $\theta_{23}> \theta_{34}$, we have $\theta_{12} < \theta_{45}, \theta_{14} < \theta_{25}.$

So, the cases (5) and (6) can be replaced by
\begin{equation*}
\begin{aligned}
\textbf{(5)} \left\{\begin{matrix}
 \theta_{13} > \theta_{35} \\
 \theta_{23} < \theta_{34} \\
 \theta_{12} > \theta_{45} \\
 \theta_{14} > \theta_{25}
\end{matrix}\right.
\quad \qquad \textbf{(6)} \left\{\begin{matrix}
  \theta_{13} < \theta_{35} \\
  \theta_{23} > \theta_{34} \\
  \theta_{12} < \theta_{45} \\
  \theta_{14} < \theta_{25}
\end{matrix}\right..
\end{aligned}
\end{equation*}

\textbf{Case (5).}

Again applying the lemma \ref{mf}, we have
\begin{equation*}
\begin{aligned}
&  m_1 \left( \frac{\sin(\theta_{12}) r_{45}^{s}\big(1-r_{12}^{s}\big)-\sin(\theta_{45}) r_{12}^{s}\big(1-r_{45}^{s}\big)}{r_{12}^{s} r_{45}^{s}}
 + \frac{\sin(\theta_{14}) r_{25}^{s}\big(1-r_{14}^{s}\big)-\sin(\theta_{25}) r_{14}^{s}\big(1-r_{25}^{s}\big)}{r_{14}^{s} r_{25}^{s}} \right) \\
&  + m_3 \frac{\sin(\theta_{34})r_{23}^{s}(1-r_{34}^{s}) - \sin(\theta_{23})r_{34}^{s}(1-r_{23}^{s})}{r_{23}^{s}r_{34}^{s}}   \\
= & m_1 \Bigg[
2\sin\frac{\theta_{45}-\theta_{12}}{2} \left( \cos\frac{\theta_{12}+\theta_{45}}{2} + \cos\frac{\theta_{14}+\theta_{25}}{2} \right)
+ \frac{1}{2^{s-1}}\left( \frac{\cos\frac{\theta_{12}}{2}}{\sin^{s-1}\frac{\theta_{12}}{2}} - \frac{\cos\frac{\theta_{45}}{2}}{\sin^{s-1}\frac{\theta_{45}}{2}} \right) \\
& + \frac{1}{2^{s-1}}\left( \frac{\cos\frac{\theta_{14}}{2}}{\sin^{s-1}\frac{\theta_{14}}{2}} - \frac{\cos\frac{\theta_{25}}{2}}{\sin^{s-1}\frac{\theta_{25}}{2}} \right)
\Bigg] \\
& + m_3 \Bigg[
2\sin\frac{\theta_{23}-\theta_{34}}{2}\cos\frac{\theta_{23}+\theta_{34}}{2}
+ \frac{1}{2^{s-1}}\left( \frac{\cos\frac{\theta_{34}}{2}}{\sin^{s-1}\frac{\theta_{34}}{2}} - \frac{\cos\frac{\theta_{23}}{2}}{\sin^{s-1}\frac{\theta_{23}}{2}} \right)
\Bigg].
\end{aligned}
\end{equation*}

By previous analysis and the monotonicity of the function $f(x)$, we have
$$\frac{\cos\frac{\theta_{12}}{2}}{\sin ^{s-1} \frac{\theta_{12}}{2}}  - \frac{\cos\frac{\theta_{45}}{2}}{\sin ^{s-1} \frac{\theta_{45}}{2}} <0,~~
\frac{\cos\frac{\theta_{14}}{2}}{\sin ^{s-1} \frac{\theta_{14}}{2}}  - \frac{\cos\frac{\theta_{25}}{2}}{\sin ^{s-1} \frac{\theta_{25}}{2}} < 0,~~ \frac{\cos\frac{\theta_{34}}{2}}{\sin ^{s-1} \frac{\theta_{34}}{2}}  - \frac{\cos\frac{\theta_{23}}{2}}{\sin ^{s-1} \frac{\theta_{23}}{2}} < 0.$$

It implies that
\begin{equation*}
\begin{aligned}
\Bigg[
&2\sin\frac{\theta_{45}-\theta_{12}}{2} \left( \cos\frac{\theta_{12}+\theta_{45}}{2} + \cos\frac{\theta_{14}+\theta_{25}}{2} \right)
+ \frac{1}{2^{s-1}}\left( \frac{\cos\frac{\theta_{12}}{2}}{\sin^{s-1}\frac{\theta_{12}}{2}} - \frac{\cos\frac{\theta_{45}}{2}}{\sin^{s-1}\frac{\theta_{45}}{2}} \right) \\
& \qquad + \frac{1}{2^{s-1}}\left( \frac{\cos\frac{\theta_{14}}{2}}{\sin^{s-1}\frac{\theta_{14}}{2}} - \frac{\cos\frac{\theta_{25}}{2}}{\sin^{s-1}\frac{\theta_{25}}{2}} \right)
\Bigg] < 0
\end{aligned}
\end{equation*}
and
\begin{equation*}
\begin{aligned}
\Bigg[
2\sin\frac{\theta_{23}-\theta_{34}}{2}\cos\frac{\theta_{23}+\theta_{34}}{2}
+ \frac{1}{2^{s-1}}\left( \frac{\cos\frac{\theta_{34}}{2}}{\sin^{s-1}\frac{\theta_{34}}{2}} - \frac{\cos\frac{\theta_{23}}{2}}{\sin^{s-1}\frac{\theta_{23}}{2}} \right)
\Bigg] < 0.
\end{aligned}
\end{equation*}
It means that $m_{1}(f_{12} + f_{14} -f_{25} - f_{45})+ m_{3}(f_{34} - f_{23}) < 0$ for positive masses. It is impossible.

\textbf{Case (6).}

Similarly, we have
\begin{equation*}
\begin{aligned}
\Bigg[
&2\sin\frac{\theta_{45}-\theta_{12}}{2} \left( \cos\frac{\theta_{12}+\theta_{45}}{2} + \cos\frac{\theta_{14}+\theta_{25}}{2} \right)
+ \frac{1}{2^{s-1}}\left( \frac{\cos\frac{\theta_{12}}{2}}{\sin^{s-1}\frac{\theta_{12}}{2}} - \frac{\cos\frac{\theta_{45}}{2}}{\sin^{s-1}\frac{\theta_{45}}{2}} \right) \\
& \qquad + \frac{1}{2^{s-1}}\left( \frac{\cos\frac{\theta_{14}}{2}}{\sin^{s-1}\frac{\theta_{14}}{2}} - \frac{\cos\frac{\theta_{25}}{2}}{\sin^{s-1}\frac{\theta_{25}}{2}} \right)
\Bigg] >0
\end{aligned}
\end{equation*}
and
\begin{equation*}
\begin{aligned}
\Bigg[
2\sin\frac{\theta_{23}-\theta_{34}}{2}\cos\frac{\theta_{23}+\theta_{34}}{2}
+ \frac{1}{2^{s-1}}\left( \frac{\cos\frac{\theta_{34}}{2}}{\sin^{s-1}\frac{\theta_{34}}{2}} - \frac{\cos\frac{\theta_{23}}{2}}{\sin^{s-1}\frac{\theta_{23}}{2}} \right)
\Bigg] > 0.
\end{aligned}
\end{equation*}
It means that $m_{1}(f_{12} + f_{14} -f_{25} - f_{45})+ m_{3}(f_{34} - f_{23}) > 0$ for positive masses. It is impossible.

In summary, we conclude that $\theta_{12} = \theta_{45}$ and $\theta_{23} = \theta_{34}$  is the unique solution of the equations \ref{keyeq}.
By $\theta_{12} = \theta_{45}$ and  $\theta_{23} = \theta_{34}$, we have $\theta_{12} = \theta_{45},$ which implies that the configuration must be symmetric.
\end{proof}

\section{Linear Stability of the Convex 1 + 5 Coorbital Central Configurations with Homogeneous Potential}

In this section, we will discuss the linear stability of the convex $1+5$ coorbital central configurations with  homogeneous potentials $s>1$.
Here, we denote the Hessian matrix of $V$ by
\begin{align*}
  H = &  \left(
    \begin{array}{cccc}
    \frac{\partial^2 V}{\partial \theta_1^2}& \frac{\partial^2 V}{\partial \theta_1 \partial \theta_2} & \cdots  & \frac{\partial^2 V}{\partial \theta_1 \partial \theta_5} \\
    \frac{\partial^2 V}{\partial \theta_2 \partial \theta_1} & \frac{\partial^2 V}{\partial \theta_2^2} & \cdots  & \frac{\partial^2 V}{\partial \theta_2 \partial \theta_5} \\
      \vdots &  \vdots & \ddots & \vdots    \\
    \frac{\partial^2 V}{\partial \theta_5 \partial \theta_1} & \frac{\partial^2 V}{\partial \theta_5 \partial \theta_2}& \cdots & \frac{\partial^2 V}{\partial \theta_5^2} \\
    \end{array}
  \right) \\
    = & \left(
    \begin{array}{cccc}
    \sum_{j\neq 1}m_1 m_j f_{1j}'& - m_1 m_2 f_{12}' & \cdots  & - m_1 m_5 f_{15}' \\
      - m_1 m_2 f_{12}' &  \sum_{j\neq 2}m_2 m_j f_{2j}' & \cdots  &  - m_2 m_5 f_{25}' \\
      \vdots &  \vdots & \ddots & \vdots    \\
      - m_1 m_5 f_{15}' & - m_2 m_5 f_{25}' & \cdots & \sum_{j\neq 5} m_j m_5 f_{j5}' \\
    \end{array}
  \right)
\end{align*}
where $f_{ij}'= f'(\theta_{ij}) = \left(\sin(\theta_{ij}) \left[ -1 + \frac{1}{2^s|\sin \frac{\theta_{ij}}{2}|^s} \right] \right)' = -\cos(\theta_{ij}) - \frac{s + (s-2) \cos(\theta_{ij})}{2^{s+1}|\sin \frac{\theta_{ij}}{2}|^s}$.

Furthermore, for convex symmetric case, we have
$$f_{12}'= f_{45}', \quad f_{13}'= f_{35}', \quad f_{14}'= f_{25}', \quad f_{15}', \quad f_{23}'= f_{34}', \quad f_{24}'.$$

It was established by Moeckel that, for a central mass sufficiently large, a relative equilibrium is linearly stable precisely when it constitutes a local minimum of the potential function $V$. The linearization of the equations of motion for coorbital configurations were obtained by Renner and Sicardy:
\begin{equation*}
\begin{aligned}
\left(\begin{array}{l}\dot{\delta\theta} \\ \dot{\delta r}\end{array}\right)=\left(\begin{array}{cc}0 & -\frac{3}{2} I_{N} \\ -2 A & 0\end{array}\right)\left(\begin{array}{l}\delta \theta \\ \delta r\end{array}\right)
\end{aligned}
\end{equation*}
where \(A=M^{-1} H\) is the acceleration coefficient matrix and \(I_{N}\) is the \(N $×$ N\) identity matrix.

They established that a coorbital relative equilibrium is linearly stable precisely when the matrix $A$ possesses no positive eigenvalues, apart from one zero eigenvalue that originates from the rotational symmetry of the system.
According to the paper \cite{Deng2022,Renner2004} we just need to prove the eigenvalues of the matrix $H$ are nonpositive. $H$ having nonpositive eigenvalues is equivalent to $-H$ having nonnegative eigenvalues.

In paper \cite{Deng2022},  for the convex symmetric $1+5$ coorbital configurations in Newtonian case, they get the simple bound.
\begin{lemma}\label{Deng} The convex symmetric $1+5$ coorbital central configurations with $0 = \theta_3  < \theta_2 < \theta_1 < \frac{\pi}{2}$ and $\theta_{15} \leq \pi$ are contained in the set $\mathcal{C}$ defined by $\frac{\pi}{6} < \theta_{2} < \frac{\pi}{3}$ and $\theta_{12} < \frac{\pi}{3}$.
\end{lemma}
Actually, above lemma is still right for homogeneous potentials $s>1$. Do not need to change anything in the proof of Lemma \ref{Deng} in paper \cite{Deng2022}.

For the linear stability of the  convex symmetric $1+5$ coorbital central configurations, we have the following result:
\begin{theorem}\label{LW} For the convex symmetric $1+5$ coorbital central configurations with positive masses, if the homogeneous potential index $s>1$, then there exists a positive constant $\mu$ such that the central configuration is linear stability when $\frac{\pi}{3} < \theta_{15} \le 2\arcsin \mu$ \emph{($2\arcsin \mu > \frac{\pi}{2}$)}.
\end{theorem}

Before proving the above theorem, we need to give the definition of \textbf{Diagonally Dominant Matrix}.

\begin{definition}
A square matrix $A$ is called \textbf{Diagonally Dominant} if $|a_{ii}| \geq \sum_{i\neq j} |a_{ij}|$ for all $i$.
\end{definition}

For a diagonally dominant matrix, there is a well-known result as follows:
\begin{lemma}\label{bf}
A symmetric diagonally dominant real matrix with nonnegative diagonal entries is positive semidefinite.
\end{lemma}

Now, we can finish the proof of the theorem \ref{LW}.

\begin{proof} We have
\[-H = \left(
    \begin{array}{cccc}
    -\sum_{j\neq 1}m_1 m_j f_{1j}'& m_1 m_2 f_{12}' & \cdots  & m_1 m_5 f_{15}' \\
      m_1 m_2 f_{12}' & -\sum_{j\neq 2}m_2 m_j f_{2j}' & \cdots  &  m_2 m_5 f_{25}' \\
      \vdots &  \vdots & \ddots & \vdots    \\
      m_1 m_5 f_{15}' & m_2 m_5 f_{25}' & \cdots & -\sum_{j\neq 5} m_j m_5 f_{j5}' \\
    \end{array}
  \right)
\]
where
$$f_{12}'= -\cos(\theta_{12}) - \frac{s + (s-2) \cos(\theta_{12})}{2^{s+1}|\sin \frac{\theta_{12}}{2}|^s}, \quad f_{13}'= -\cos(\theta_{13}) - \frac{s + (s-2)\cos(\theta_{13})}{2^{s+1}|\sin \frac{\theta_{13}}{2}|^s},$$
$$ f_{14}'= -\cos(\theta_{14}) - \frac{s + (s-2)\cos(\theta_{14})}{2^{s+1}|\sin \frac{\theta_{14}}{2}|^s}, \quad f_{15}'= -\cos(\theta_{15}) - \frac{s + (s-2) \cos(\theta_{15})}{2^{s+1}|\sin \frac{\theta_{15}}{2}|^s},$$
$$f_{23}'= -\cos(\theta_{23}) - \frac{s + (s-2) \cos(\theta_{23})}{2^{s+1}|\sin \frac{\theta_{23}}{2}|^s}, \quad f_{24}'= -\cos(\theta_{24}) - \frac{s + (s-2) \cos(\theta_{24})}{2^{s+1}|\sin \frac{\theta_{24}}{2}|^s},$$
$$f_{25}'= -\cos(\theta_{25}) - \frac{s + (s-2) \cos(\theta_{25})}{2^{s+1}|\sin \frac{\theta_{25}}{2}|^s}, \quad f_{34}'= -\cos(\theta_{34}) - \frac{s + (s-2) \cos(\theta_{34})}{2^{s+1}|\sin \frac{\theta_{34}}{2}|^s}, $$
$$f_{35}'= -\cos(\theta_{35}) - \frac{s + (s-2) \cos(\theta_{35})}{2^{s+1}|\sin \frac{\theta_{35}}{2}|^s}, \quad  f_{45}'= -\cos(\theta_{45}) - \frac{s + (s-2) \cos(\theta_{45})}{2^{s+1}|\sin \frac{\theta_{45}}{2}|^s}.$$

Here, it is easy to see that the matrix is diagonally dominant if $f_{15}'$ is negative. Since $f_{15}' < 0$ when $\theta_{15} \in (0, \frac{\pi}{2})$, we just need to consider the case $\theta_{15} \in [\frac{\pi}{2}, \pi)$.

$$f_{15}' = -\cos(\theta_{15}) - \frac{s + (s-2)\cos(\theta_{15})}{2^{s+1}\sin^s \frac{\theta_{15}}{2}} = - \frac{ 2^{s+1}\sin^s \frac{\theta_{1N}}{2} \cos (\theta_{15})  + s + (s-2)\cos (\theta_{15})}{2^{s+1}\sin^s \frac{\theta_{15}}{2}} $$

So the numerator of $f_{15}'$ is
\begin{align*}
    & 2^{s+1}\sin^s \frac{\theta_{15}}{2} \cos (\theta_{15})  + s + (s-2)\cos (\theta_{15}) \\
  = & 2^{s+1}\sin^s \frac{\theta_{15}}{2} (1 - 2\sin^2\frac{\theta_{15}}{2}) + s + (s-2) (1 - 2\sin^2 \frac{\theta_{15}}{2}) \\
  = & 2\big[-2^{s+1} \sin^{s+2} \frac{\theta_{15}}{2} + 2^s \sin^s \frac{\theta_{15}}{2} - (s-2)\sin^2 \frac{\theta_{15}}{2} + s-1) \big].
\end{align*}
Let $f(t) = -2^{s+1} t^{s+2} + 2^s t^s - (s-2)t^2 + s-1 $, where $t = \sin \frac{\theta_{15}}{2}$ and $ t \in [\frac{\sqrt{2}}{2}, 1]$.

It is easy to check that $f(\frac{\sqrt{2}}{2}) = \frac{s}{2} > 0$ and $f(1)=-2^s +1 < 0$ when $s > 1$.

Since the function $f(t)$ is strictly monotonically decreasing for $t \in[\frac{\sqrt{2}}{2}, 1]$ when $s>1$, we can find only one  number $t = \mu > \frac{\sqrt{2}}{2},$ which depends on $s$, make $f(\mu)=0$. For example, we have $\mu \approx 0.810815$  for $s=3$ and $\mu \approx 0.8264$ for $s=2$.  It means that $f(t)\ge 0$ when $t \in[\frac{\sqrt{2}}{2}, \mu]$.

Let $\sin \frac{\theta_{15}}{2}=\mu$ , we have $\theta_{15}=2\arcsin \mu > \frac{\pi}{2}$.

Therefore, there exists a positive constant $\mu$ such that $f_{15}^{\prime}$ is nonpositive when $\frac{\pi}{3}<\theta_{15} \le 2\arcsin \mu$. It implies that the matrix $-H$ is diagonally dominant when $\frac{\pi}{3}<\theta_{15} \le 2\arcsin \mu$.

By the lemma \ref{bf}, we conclude that the matrix $-H$ positive semidefinite, which implies that the eigenvalues of the matrix H are nonpositive. It means that the convex symmetric $1+5$ coorbital central configuration is linearly stability when $\frac{\pi}{3}<\theta_{15}\le 2\arcsin \mu$.
\end{proof}
Furthermore we find that the above result can be extended to any convex $1+N$ coorbital central configurations. We do not need the symmetric condition.
\begin{theorem}\label{LS} For the convex $1+N$ coorbital central configurations with positive masses, if the homogeneous potential index $s>1$, there exists a positive constant $\mu$ such that the central configuration is linear stability when $0 < \theta_{1N} \le 2\arcsin \mu$ \emph{($2\arcsin \mu > \frac{\pi}{2}$)}.
\end{theorem}

\begin{proof} For $1+N$ coorbital problem, the Hessian matrix of $V$ is
\[H = \left(
    \begin{array}{cccc}
    \sum_{j\neq 1}m_1 m_j f_{1j}'& -m_1 m_2 f_{12}' & \cdots  & -m_1 m_N f_{1N}' \\
      -m_1 m_2 f_{12}' & \sum_{j\neq 2}m_2 m_j f_{2j}' & \cdots  &  -m_2 m_N f_{2N}' \\
      \vdots &  \vdots & \ddots & \vdots    \\
      -m_1 m_N f_{1N}' & -m_2 m_N f_{2N}' & \cdots & \sum_{j\neq N} m_j m_N f_{jN}' \\
    \end{array}
  \right)
\]
where $f_{ij}' = -\cos(\theta_{ij}) - \frac{s + (s-2) \cos(\theta_{ij})}{2^{s+1}|\sin \frac{\theta_{ij}}{2}|^s}.$

The similar to the Theorem \ref{LS}, we need to prove that the eigenvalues of $-H$ are nonnegative under our assumptions.

Here, we see that the matrix is diagonally dominant if $f_{1N}'$ is negative. Since $f_{1N}' < 0$ when $\theta_{1N} \in (0, \frac{\pi}{2})$, we just need to consider $\theta_{1N} \in [\frac{\pi}{2}, \pi)$.

For
$$f_{1N}' = -\cos(\theta_{1N}) - \frac{s + (s-2) \cos(\theta_{1N})}{2^{s+1}\sin^s \frac{\theta_{1N}}{2}} = - \frac{ 2^{s+1}\sin^s \frac{\theta_{1N}}{2} \cos (\theta_{1N})  + s + (s-2)\cos (\theta_{1N})}{2^{s+1}\sin^s \frac{\theta_{1N}}{2}},$$
the numerator of $f_{1N}'$ is
\begin{align*}
    & 2^{s+1}\sin^s \frac{\theta_{1N}}{2} \cos (\theta_{1N})  + s + (s-2)\cos (\theta_{1N}) \\
  = & 2\big[-2^{s+1} \sin^{s+2} \frac{\theta_{1N}}{2} + 2^s \sin^s \frac{\theta_{1N}}{2} - (s-2)\sin^2 \frac{\theta_{1N}}{2} + s-1\big].
\end{align*}
Let $f(t) = -2^{s+1} t^{s+2} + 2^s t^s - (s-2)t^2 + s-1 $, where $t = \sin \frac{\theta_{1N}}{2}$ and $ t \in [\frac{\sqrt{2}}{2}, 1]$.
We have $f(\frac{\sqrt{2}}{2}) = \frac{s}{2} > 0$ and $f(1)=-2^s +1 < 0$ for $s > 1$.

Since the function $f(t)$ is strictly monotonically decreasing for $t \in[\frac{\sqrt{2}}{2}, 1]$ when $s>1$, we can find only one  number $t = \mu > \frac{\sqrt{2}}{2}$ make $f(\mu)=0$, which implies that $f(t)\ge 0$ when $t \in[\frac{\sqrt{2}}{2}, \mu]$.

Let $\sin \frac{\theta_{1N}}{2}=\mu$ , we have $\theta_{1N}=2\arcsin \mu > \frac{\pi}{2}$.

Therefore, there exists a positive constant $\mu$ such that $f_{1N}^{\prime}$ is nonpositive when $0 <\theta_{1N} \le 2\arcsin \mu$. It implies that the matrix $-H$ is diagonally dominant when $0 <\theta_{1N} \le 2\arcsin \mu$.

By the lemma \ref{bf}, we conclude that the matrix $-H$ positive semidefinite, which implies that the eigenvalues of the matrix H are nonpositive. So, the convex $1+N$ coorbital central configuration is linear stability when $0 <\theta_{1N}\le 2\arcsin \mu$ for $s>1$.
\end{proof}

\phantomsection

\end{document}